\documentclass[12pt]{amsart}

\usepackage{mathrsfs}
\usepackage{amsmath,amssymb,amsfonts,latexsym,txfonts}     
\usepackage{eucal}
\usepackage{fancyhdr}
\usepackage{graphicx}

\newlength{\originalbase}
\newtheorem{theorem}{Theorem}[section]
\newtheorem*{maintheorem}{Main Theorem}
\newtheorem{proposition}[theorem]{Proposition}

\newtheorem{corollary}[theorem]{Corollary}

\newtheorem*{conjecture}{Conjecture}
\newtheorem{question}{Open Question}

\input amssym.def
 
\input amssym.tex

\newcommand{\RR}{\mathbb{R}}
\newcommand{\KK}{\mathcal{K}}
\newcommand{\LL}{\mathcal{L}}

\begin{document}

\title{Volume bounds for shadow covering}
\author[C. Chen]{Christina Chen} 
\address{C. Chen, Newton North High School, 
Newton, MA 02460}
\author[T. Khovanova]{Tanya Khovanova} 
\address{T. Khovanova,
Department of Mathematics,
MIT,
77 Massachusetts Avenue,
Cambridge, MA 02139}
\author[D. Klain]{Daniel A. Klain} 
\address{D. Klain, Department of Mathematical Sciences,
University of Massachusetts Lowell,
Lowell, MA 01854, USA}
\email{Daniel\_{}Klain@uml.edu}

\begin{abstract} 
For $n \geq 2$ a construction is given for a large family of compact 
convex sets $K$ and $L$ in $\RR^n$ such that the orthogonal projection 
$L_u$ onto the subspace $u^\perp$ contains a translate of 
the corresponding projection $K_u$ 
for every direction $u$, while the volumes of $K$ and $L$ satisfy $V_n(K) > V_n(L).$

It is subsequently shown that, if the orthogonal projection 
$L_u$ onto the subspace $u^\perp$ contains a translate of $K_u$ 
for every direction $u$, then the set $\frac{n}{n-1} L$ contains
a translate of $K$.  If follows that
$$V_n(K) \leq \left(\frac{n}{n-1}\right)^n V_n(L).$$
In particular, we derive a {\em universal constant bound}
$$V_n(K) \leq 2.942 \, V_n(L),$$
independent of the dimension $n$ of the ambient space.
Related results are obtained for projections onto
subspaces of some fixed intermediate co-dimension.
Open questions and conjectures are also posed.
\end{abstract}


\maketitle

\section{Introduction}

Suppose that $K$ and $L$ are compact convex subsets of 
$n$-dimensional Euclidean space.
For a given fixed dimension $1 \leq k < n$,
suppose that every $k$-dimensional orthogonal projection (shadow) of $K$
can be translated inside the corresponding projection of $L$.  
How are the volumes of $K$ and $L$ related?  And under what 
additional conditions does it follow that $L$ contains a translate of $K$?

Several aspects of this problem have been recently addressed in \cite{Klain-Inscr,Klain-Shadow,Klain-Circ}.
In \cite{Klain-Shadow} it was shown that, despite the assumption
on covering by all $k$-dimensional projections, it may still be the case that $K$ has {\em greater}
volume than $L$.  

It is also shown in \cite{Klain-Shadow} that, if the orthogonal projection $L_u$ of $L$
onto the $(n-1)$-dimensional subspace $u^\perp$ contains a translate of the corresponding
projection $K_u$ for every unit direction $u \in \RR^n$, then 
the volumes must satisfy
$V_n(K) \leq n V_n(L)$, and that $V_n(K) \leq V_n(L)$ if $L$ can be approximated by Blaschke combinations
of convex cylinders in $\RR^n$.  
Earlier results of Ball \cite{Ball-Shad} imply that the covering
condition on projections (as well as the much weaker condition that projections of $K$ have
smaller area) imply that the volume ratio $\tfrac{V_n(K)}{V_n(L)}$ is bounded by a function that grows
with order $\sqrt{n}$ as the dimension $n$ becomes large.  However, all specific examples so far
computed have suggested that the volume ratio is much smaller.  

In this article we prove that
the volume of $K$, while possibly exceeding that of $L$, must still always
satisfy
$$V_n(K) \leq \left(\frac{n}{n-1}\right)^n V_n(L).$$
In particular, there is a {\em universal constant bound}
\begin{align}
V_n(K) \leq 2.942 \, V_n(L),
\label{univconst}
\end{align}
independent of the dimension $n$ of the ambient space 
(see Section~\ref{sec.general}).

This constant bound will be seen as the direct consequence of the main theorem of
this article:
\begin{maintheorem} Let $K$ and $L$ be compact convex sets in $\RR^n$.
Suppose that, for every unit vector $u$, the 
orthogonal projection $L_u$ of $L$ onto the subspace $u^\perp$
contains a translate
of the corresponding projection $K_u$.  Then there exists $x \in \RR^n$ such that

$$K + x \, \subseteq \, \Big(\tfrac{n}{n-1}\Big) L.$$
\label{main}
\end{maintheorem}
We also provide a substantial source of examples of compact convex sets $K$ and $L$
such that the volume ratio is strictly greater than 1, adding to the special case
described in \cite{Klain-Shadow}.

The background material for these results is described in Section~\ref{sec.backg}. 
In Section~\ref{sec.interpolate} we describe a large family of convex bodies 
$K$ and $L$ such that each projection of $L$ contains a translate of the corresponding
projection of $K$, while $K$ has greater volume.

In Section~\ref{sec.simplex} we show that if the body $L$ having larger projections is a {\em simplex}
then there is a translate of $K$ that lies inside a cap body of $L$ having volume $\tfrac{n-1}{n}V_n(L)$.
Section~\ref{sec.general} combines the simplicial case with a containment theorem of Lutwak 
\cite{Lut-contain} (see also \cite[p.~54]{Lincee}) to prove the Main Theorem.  
The universal constant bound~(\ref{univconst}) for volume ratios is
then derived as a corollary.

Section~\ref{sec.dim} extends the results of the previous sections to the case in which $(n-d)$-dimensional
projections of $L$ contain translates of $(n-d)$-dimensional projections of $K$ for some intermediate co-dimension
$1 \leq d \leq n-1$.  In Section~\ref{sec.conclusion} we pose some open questions and conjectures.

This investigation is motivated in part
by the projection theorems of Groemer \cite{Gro-proj}, Hadwiger \cite{Had-proj}, 
and Rogers \cite{Rogers}.  In particular, if two compact convex sets
have translation congruent (or, more generally, homothetic) projections in every linear subspace of some chosen
dimension $k \geq 2$, then the original sets $K$ and $L$ must be translation congruent (or homothetic).
Rogers also proved analogous results for sections of sets with hyperplanes through a base point \cite{Rogers}.
These results then set the stage for more general (and often much more difficult) questions, in which 
the rigid conditions of translation congruence or homothety are 
replaced with weaker conditions, such as containment up to translation, inequalities of measure, etc.

Progress on the general question of when one convex body must contain a translate (or a congruent copy) 
of another appears in the work of
Gardner and Vol\v{c}i\v{c} \cite{Gard-Vol},
Groemer \cite{Gro-proj},
Hadwiger \cite{Had3,Had4,Had-proj,Lincee,Santa}, 
Jung \cite{Bonn2,Webster}, 
Lutwak \cite{Lut-contain}, 
Rogers \cite{Rogers}, 
Soltan \cite{Soltan},
Steinhagen \cite[p. 86]{Bonn2}, 
Zhou \cite{Zhou1,Zhou2}, and others (see also \cite{Gard2006}).
The connection between projections or sections of convex bodies and comparison of their volumes 
also lies at the heart of each of  
two especially notorious inverse problems: the Shephard Problem \cite{Shep}
(solved independently by  
Petty \cite{Petty-shep} and Schneider \cite{Schneider-shep}) and
the even more difficult Busemann-Petty Problem \cite{bus-pet}
(see, for example,
\cite{Ball-BP,Bourgain,Gard-BP,Gard-Busemann,Gard-Kold,Giannopoulos,Koldobsky-Busemann,Larman-Rogers,Lut1988,Papa,Zhang-Busemann-old,Zhang-Busemann}).
A more complete discussion of these and related problems (many of which remain open) can be
found in the comprehensive book by Gardner \cite{Gard2006}.

\section{Background}
\label{sec.backg} 
We will require several concepts and established results from convex geometry in Euclidean space.
Denote $n$-dimensional Euclidean space by $\RR^n$, and let $\KK_n$ denote 
the set of compact convex subsets of $\RR^n$.  The $n$-dimensional
(Euclidean) volume of a set $K \in \KK_n$ will be denoted $V_n(K)$.  If $u$ is a unit vector in
$\RR^n$, denote by $K_u$ the orthogonal projection of a set $K$ onto the subspace $u^\perp$.
The boundary of a compact
convex set $K$ will be denoted by $\partial K$.

Let $h_K: \RR^n \rightarrow \RR$ denote the support function of a compact convex set $K$;
that is,
$$h_K(v) = \max_{x \in K} x \cdot v.$$
The standard separation theorems of convex geometry imply that
the support function $h_K$ characterizes the body $K$; that is, $h_K = h_L$ if and only if $K=L$.
If $K_i$ is a sequence in $\KK_n$, then $K_i \rightarrow K$ in the Hausdorff topology if and only if
$h_{K_i} \rightarrow h_K$ uniformly when restricted to the unit sphere in $\RR^n$.

If $u$ is a unit vector in
$\RR^n$, denote by $K^u$ the support set of $K$ in the direction of $u$; that is,
$$K^u = \{x \in K \; | \; x \cdot u = h_K(u) \}.$$
If $P$ is a convex polytope, then $P^u$ is the face of $P$ having $u$ in its outer normal cone.

Given two
compact convex sets $K, L \in \KK_n$ and $a,b \geq 0$ denote
$$aK + bL = \{ax + by \; | \; x \in K \hbox{ and } y \in L\}.$$
An expression of this form is called a {\em Minkowski combination} or 
{\em Minkowski sum}.  Because $K$,~$L \in \KK_n$, the set $aK + bL \in \KK_n$ as well.  
Convexity of $K$ also implies that $aK + bK = (a+b)K$ for all $a,b \geq 0$.
Support functions are easily shown to satisfy the identity
$h_{aK+bL} = ah_K + bh_L$.  

The volume of a Minkowski combination satisfies
a concavity property called the {\em Brunn-Minkowski inequality}.  Specifically,
for $0 \leq t \leq 1$,
\begin{equation}
V_n((1-t)K + t L)^{1/n} \geq (1 - t)V_n(K)^{1/n} + t V_n(L)^{1/n}.
\label{bmcc}
\end{equation}
If $K$ and $L$ have interior, then equality holds in~(\ref{bmcc}) if and only if $K$ and $L$
are homothetic; that is, 
iff there exist $a \in \RR$ and $x \in \RR^n$ such that $L = aK + x$.    
See, for example, any of 
\cite{Bonn2,Gard-BM,red,Webster}.

The volume $V_n(aK + bL)$ is explicitly given by {\em Steiner's formula:}
\begin{equation}
V_n(aK + bL) = \sum_{i=0}^n \binom{n}{i} \, a^{n-i} b^{i} 
V_{n-i, i}(K, L),
\label{steinform}
\end{equation}
where the {\em mixed volumes} 
$V_{n-i, i}(K, L)$ depend only on $K$ and $L$ and the indices $i$ and $n$.
In particular, if we fix two convex sets $K$ and $L$ then
the function $f(a,b) = V_n(aK + bL)$ is a homogeneous polynomial of degree $n$
in the non-negative variables $a, b$.  

Each mixed volume $V_{n-i, i}(K, L)$ is non-negative, continuous in
the entries $K$ and $L$, and monotonic with respect to set inclusion.
Note also that $V_{n-i, i}(K, K) = V_n(K)$.
If $\psi$ is an affine transformation whose linear component has determinant 
denoted $\det \psi$, then 
$V_{n-i, i}(\psi K, \psi L) = |\det \psi| \, V_{n-i, i}(K, L)$.
It also follows from~(\ref{steinform}) that $V_{n-i,i}(aK,bL) = a^{n-i}b^i V_{n-i,i}(K,L)$ 
for all $a,b \geq 0$.

If $P$ is a polytope, then the mixed volume $V_{n-1, 1}(P, K)$
satisfies the classical ``base-height" formula 
\begin{equation}
V_{n-1, 1}(P, K) = \frac{1}{n} \sum_{u \perp \partial P} h_K(u) V_{n-1}(P^u),
\label{polyvol}
\end{equation}
where this sum is finite, taken over all outer normals $u$ to the {\em facets} 
on the boundary $\partial P$.
These and many other properties of convex bodies and mixed volumes 
are described in detail in each of \cite{Bonn2,red,Webster}.

The identity~(\ref{polyvol}) implies the following useful containment theorem for simplices, due to Lutwak \cite{Lut-contain}.
\begin{theorem} Let $K \in \KK_n$ and let $\triangle$ be an $n$-dimensional simplex.
Then $\triangle$ contains a translate of $K$ if and only if 
$$V_{n-1,1}(\triangle,K) \leq V_n(\triangle).$$ 
\label{lut}
\end{theorem}
\begin{proof} Since mixed volumes are translation invariant and monotonic with respect to inclusion of sets,
it is immediate that $V_{n-1,1}(\triangle,K) \leq V_n(\triangle)$ whenever $\triangle$ contains a translate of $K$.

Conversely, suppose that $V_{n-1,1}(\triangle,K) \leq V_n(\triangle)$.  Evidently $\triangle$ contains a translate of $K$
if $K$ is a single point, so let us assume that $K$ is not a  single point, 
so that $V_{n-1,1}(\triangle,K) >  0$. (See, for example, \cite[p. 277]{red}.)

Let $\alpha > 0$ be maximal such that $\triangle$ contains a translate of $\alpha K$.  
Without loss of generality, assume $\alpha K \subseteq \triangle$.  
If $\alpha K$ does not meet every facet of $\triangle$, 
then $\alpha K$ can be translated (in the direction of the unit normal to the untouched facet) 
into the interior of $\triangle$, violating the maximality of $\alpha$.  
Therefore, $\alpha K$ meets each facet of $\triangle$, so that
$$h_{\alpha K}(u) = h_\triangle(u)$$
for each unit normal $u$ to facets of $\triangle$.  The formula~(\ref{polyvol}) now yields,
$$V_{n-1,1}(\triangle,\alpha K) 
= \frac{1}{n} \sum_{u \perp \partial \triangle} h_{\alpha K}(u) V_{n-1}(\triangle^u)
= \frac{1}{n} \sum_{u \perp \partial \triangle} h_{\triangle}(u) V_{n-1}(\triangle^u) = V_n(\triangle),$$
so that
$$\alpha V_{n-1,1}(\triangle, K) = V_{n-1,1}(\triangle,\alpha K) = V_n(\triangle) \geq  V_{n-1,1}(\triangle,K).$$
Since  $V_{n-1,1}(\triangle,K)>0$, it follows that $\alpha \geq 1$, so that $\triangle$ contains 
a translate of $K$.
\end{proof}

Suppose that $\mathscr{F}$ is a family of compact convex sets in $\RR^n$.
Helly's Theorem \cite{Bonn2,red,Webster} asserts that, if every $n+1$ sets
in $\mathscr{F}$ share a common point, then the entire family shares a common point.
In \cite{Lut-contain} Lutwak used Helly's theorem to
prove the following fundamental criterion for whether a set $L \in \KK_n$ contains 
a translate of another set $K \in \KK_n$.  
\begin{theorem}[Lutwak's Containment Theorem]  Let $K, L \in \KK_n$.  
Suppose that, for every $n$-simplex $\triangle$ such that $L \subseteq \triangle$, 
there is a vector $v_\triangle \in \RR^n$ such that $K + v_\triangle \subseteq \triangle$.  
Then there is a vector $v \in \RR^n$ such that $K +v \subseteq L$.
\label{lcon}
\end{theorem}
A proof of this containment theorem is also given in \cite[p.~54]{Lincee}.  
We will make use of
this result in Section~\ref{sec.general}.
Variations of Theorem~\ref{lcon} in which circumscribing simplices 
are replaced 
with inscribed simplices or circumscribing cylinders are proved in 
\cite{Klain-Inscr} and \cite{Klain-Circ} respectively.  

Theorem~\ref{lcon} has the following immediate consequence.
\begin{proposition} Let $K \in \KK_n$.  Then $-nK$ contains a translate of $K$.
\label{KnK}
\end{proposition}

\begin{proof} Suppose that $\triangle$ is an $n$-dimensional simplex such that 
$-nK \subseteq \triangle$.
It follows that 
$K \subseteq -\tfrac{1}{n}\triangle.$

Meanwhile, since $\triangle$ is a $n$-dimensional simplex, the centroids of its facets 
are the vertices of a translate of $-\frac{1}{n}\triangle$.  It follows that 
$\triangle$ contains a translate of $K$.
Since this holds for every simplex $\triangle$ that contains $-nK$, it follows
from Theorem~\ref{lcon} that $-nK$ contains a translate of $K$.
\end{proof}

\section{Interpolating with a simplex}
\label{sec.interpolate}

If a convex body $K$ in $\RR^n$ has positive volume, then $K$ has at least $n+1$ exposed points \cite[p. 89]{Webster}.
It follows from \cite[Theorem~2.4]{Klain-Inscr}
that there exists a simplex $\triangle$ such that every projection $\triangle_u$ contains a translate of 
the corresponding projection $K_u$, 
while $\triangle$ does {\em not} contain a translate of $K$.  

In general, under these shadow covering conditions, either of the bodies $K$  or $\triangle$ 
may possibly have larger volume.  However,
the next theorem asserts that there is always a convex Minkowski combination of $K$ and $\triangle$
that ``hides behind" $\triangle$, while having {\em larger volume}. (See, for example, Figure~\ref{pbody}.)
\begin{theorem} Suppose that $\triangle$ is an $n$-simplex, and $K$ is a compact convex set in $\RR^n$ such that
the following assertions hold:
\begin{enumerate}
\item[(i)] {\rm Each projection $\triangle_u$ contains a translate of the corresponding projection $K_u$.} \\[-5mm]
\item[(ii)] {\rm The simplex $\triangle$ does not contain a translate of $K$.}
\label{test}
\end{enumerate}
Then there exists $t \in (0,1)$ and a convex body $L = (1-t)K + t\triangle$ such that the following assertions hold:
\begin{enumerate}
\item[(i)$'$]  {\rm Each projection $\triangle_u$ contains a translate of the corresponding projection $L_u$. }\\[-5mm]
\item[(ii)$'$]   $V_n(L) > V_n(\triangle)$.
\end{enumerate}
\label{interp}
\end{theorem}

\begin{proof}  Suppose that $t \in [0,1]$, that $L = (1-t)K + t\triangle$, and that $u$ is a unit vector.
We are given in (i) that $\triangle_u$ contains a translate of $K_u$, so that $K_u + w \subseteq \triangle_u$ for some
vector $w \in u^\perp$.  It follows that
\begin{align*}
L_u + (1-t)w 
& = (1-t)K_u + t\triangle_u +  (1-t)w \\
& = (1-t)(K_u + w)+ t\triangle_u \\
& \subseteq (1-t)\triangle_u +  \, t\triangle_u \\
& = \triangle_u,
\end{align*}
so that $\triangle_u$ contains a translate of $L_u$ as well.  This verifies (i)$'$ for all $t \in [0,1]$.

Next, we find a value of $t$ so that (ii)$'$ holds.
For $t \in [0,1]$, define 
$$f(t) = V_n \Big( (1-t)K + t\triangle \Big).$$
Steiner's formula~(\ref{steinform}) implies that $f$ has the polynomial expansion 
$$f(t) = \sum_{i=0}^n \binom{n}{i} V_{i, n-i}(K,\triangle) (1-t)^i t^{n-i},$$
so that
$$f'(t) = \sum_{i=0}^n \binom{n}{i} V_{i, n-i}(K,\triangle) [-i(1-t)^{i-1} t^{n-i} + (n-i)(1-t)^i t^{n-i-1}].$$
It follows that
$$f'(1) = nV_{0, n}(K,\triangle) -nV_ {1,n-1}(K,\triangle) = nV_n(\triangle) - nV_ {1,n-1}(K,\triangle).$$

From the symmetry of mixed volumes, we have $V_ {1,n-1}(K,\triangle) = V_ {n-1,1}(\triangle, K)$.
Since $\triangle$ does not contain a translate of $K$, Theorem~\ref{lut} now implies that 
$$V_ {1,n-1}(K,\triangle) = V_ {n-1,1}(\triangle, K) > V_n(\triangle),$$
so that $f'(1) < 0$.  It follows that $f(t) > f(1)$ for some $t \in (0,1)$.  Setting 
$L = (1-t)K + t\triangle$ for this value of $t$ completes the proof.
\end{proof}

\begin{figure}[ht]
\includegraphics[scale=0.1]{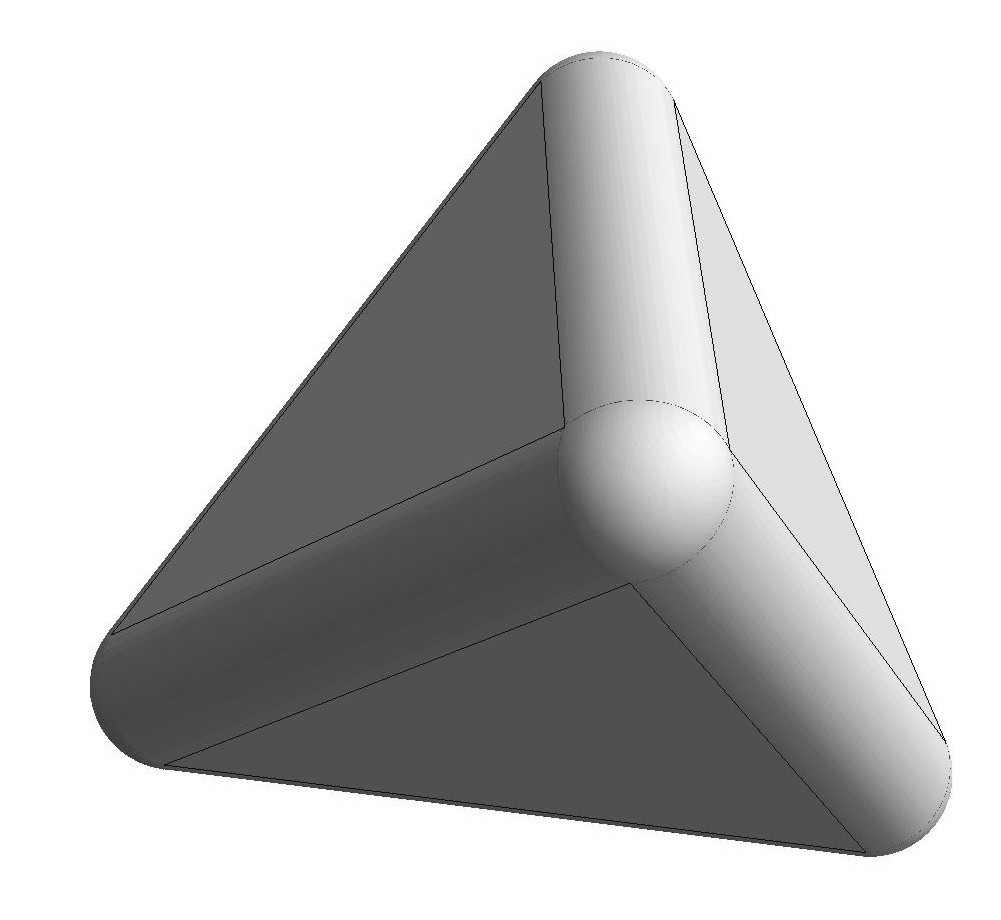} 
\caption{A convex Minkowski combination of a regular tetrahedron with a Euclidean ball.}
\label{pbody}
\end{figure}

Theorem~\ref{interp}, together with \cite[Theorem~2.4]{Klain-Inscr}, implies
the following corollary.
\begin{corollary}  Suppose that $K \in \KK_n$ has positive volume.  Then there exists a simplex
$\triangle$ and $t \in (0,1)$ such that every projection of $\triangle$ contains a translate of
the corresponding projection of the body
$L = (1-t)\triangle + tK,$
while $V_n(L) > V_n(\triangle)$. 
\label{posvol}
\end{corollary}

More generally,
suppose that $K_0, K_1 \in \KK_n$ such that every projection of $K_1$ contains a
translate of the corresponding projection of $K_0$, while $K_1$ does not contain $K_0$.  
If $t \in (0,1)$, the
interpolated body
$$K_t = (1-t)K_0 + tK_1$$
also satisfies these conditions.  Theorem~\ref{interp} motivates the following question:
Under what conditions on $K_0$ and $K_1$ does there exist $t$ so that
\begin{align}
V_n(K_t) > V_n(K_1)?
\label{volineq}
\end{align}
Theorem~\ref{interp} implies there exists such a value $t$ if $K_1$ is a simplex.  
On the other hand, there are large classes of convex bodies $K_1$ for which
no such $t$ exists.  For example, if $K_1$ is a centrally symmetric body then
the Brunn-Minkowski inequality~(\ref{bmcc}) can be used to show that~(\ref{volineq})  will not hold 
(see, for example, \cite{Klain-Shadow}).

More generally, a set $L \in \KK_n$ is called a {\em cylinder body} if $L$
can be expressed as a limit of Blaschke combinations of 
cylinders (see \cite{Klain-Shadow}).  Here a {\em cylinder} refers to
the Minkowski sum in $\RR^n$ of an $(n-1)$-dimensional convex body with a line segment.

It turns out that if $K_1$ is a cylinder body, 
then no $t \in (0,1)$ will satisfy~(\ref{volineq}), as the 
next proposition explains.
\begin{proposition} Suppose that $K_0, K_1 \in \KK_m$ such that every projection of $K_1$ contains a
translate of the corresponding projection of $K_0$, while $K_1$ does not contain $K_0$.  
If $K_1$ is an $(n-1)$-cylinder body, then
$$V_n(K_0) \leq V_n(K_t) \leq V_n(K_1)$$
for all $t \in (0,1)$.
\end{proposition}
\begin{proof} If $K_1$ is $(n-1)$-cylinder body, then 
$$V_n(K_0) \leq V_n(K_1) \quad \hbox{ and } \quad  V_n(K_t) \leq V_n(K_1),$$ 
by \cite[Theorem~6.1]{Klain-Shadow}.
It follows from the Brunn-Minkowski inequality~(\ref{bmcc}) that, for $t \in (0,1)$,
\begin{align*}
V_n(K_t)^{1/n} & = V_n\Big( (1-t)K_0 + tK_1 \Big)^{1/n} \\ 
& \geq (1-t)V_n(K_0)^{1/n} + tV_n(K_1)^{1/n} \\
& \geq V_n(K_0)^{1/n},
\end{align*}
so that $V_n(K_0) \leq V_n(K_t)$ as well.
\end{proof}

\section{When a body can hide behind a simplex}
\label{sec.simplex}

Denote by 
$\Xi$ the
$n$-dimensional simplex having vertices at $\{o, e_1, \ldots, e_n\}$,
where each $e_i$ is the $i$-th coordinate unit vector of $\RR^n$,
and $o$ is the origin.
The simplex $\Xi$ has outer facet unit normals given by
$$\{-e_1, -e_2, \ldots, -e_n, v \},$$
where
$v \in \RR^n$ is the unit vector with coordinates
$$v = \left( \tfrac{1}{\sqrt{n}}, \tfrac{1}{\sqrt{n}}, \ldots, \tfrac{1}{\sqrt{n}} \right).$$
Note also that each $\Xi_{e_i} \subseteq \Xi$, being the 
$(n-1)$-dimensional simplex having vertices $\{o, e_1, \ldots, e_{i-1}, e_{i+1}, \ldots, e_n\}$.

Let $D$ denote the cap body formed by the convex hull of $\Xi$
with the point $p \in \RR^n$ having coordinates
$$p = \left(\tfrac{1}{n-1}, \tfrac{1}{n-1}, \ldots, \tfrac{1}{n-1} \right).$$
See Figure~\ref{capbody}.  
For each $i$ let $w_i \in \RR^n$ denote the vector with coordinates
$$w_i = \left( 1, \ldots, 1, 0, 1 \ldots, 1 \right),$$
where the $0$ appears in the $i$th coordinate.
We can represent $D$ as the intersection of half-spaces
$$D = \left( \bigcap_{i=1}^n \{x \; | \;\; e_i \cdot x \geq 0 \} \right) \cap 
\left( \bigcap_{i=1}^n \{ x \; | \;\; w_i \cdot x \leq 1 \} \right).$$
\begin{figure}[ht]
\hspace{0mm}\includegraphics[scale=0.1]{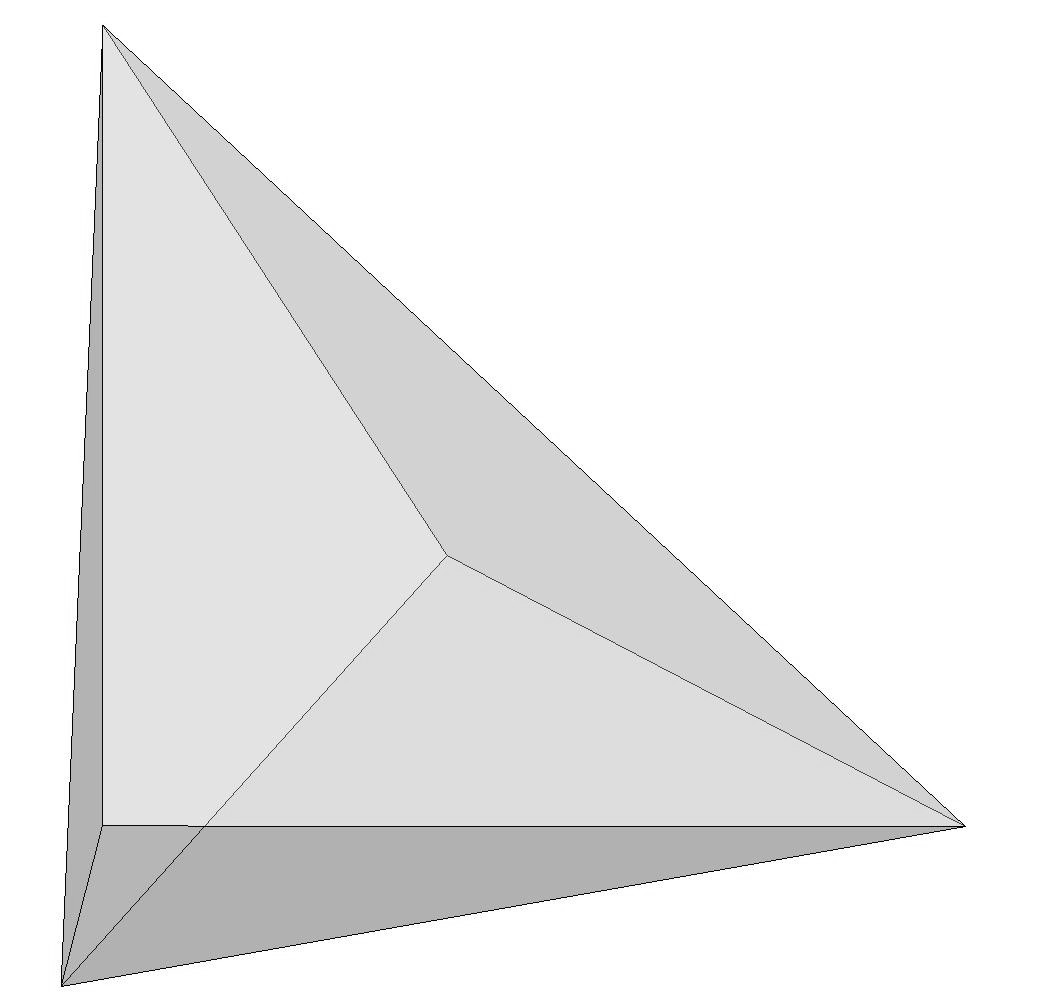}\hspace{5mm} 
\includegraphics[scale=0.12]{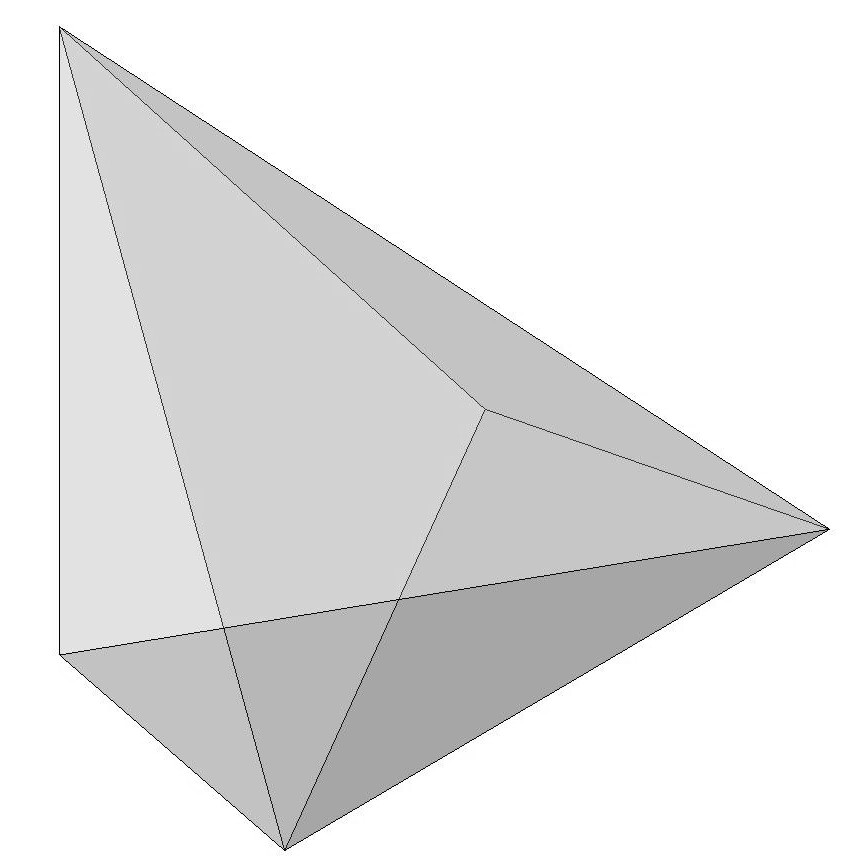}\hspace{9mm} 
\includegraphics[scale=0.14]{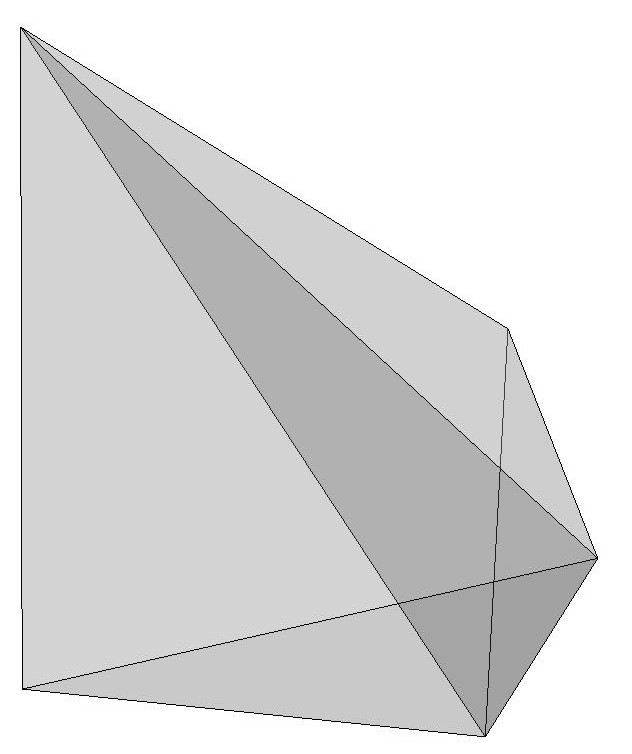} 
\caption{Three views of the cap body $D$ (3-dimensional case).}
\label{capbody}
\end{figure}

For each $i$, let $E_i$ denote the line segment with endpoints at $o$ and $e_i$, and let
$C_i = \Xi_{e_i} + E_i$, a prism (i.e. a cylinder with a simplicial base) that contains $\Xi$.  
We can represent each $C_i$ as the intersection of half-spaces
$$C_i = \left(\bigcap_{i=1}^n \{x \; | \;\; e_i \cdot x \geq 0 \}\right) 
\cap \{x \; | \;\; e_i \cdot x \leq 1 \} \cap \{x \; | \;\; w_i \cdot x \leq 1 \}.$$
It follows that
$$D = \bigcap_{i=1}^n C_i.$$

The next result is fundamental.
\begin{theorem} Let $K \in \KK_n$, and suppose that, 
for every unit vector $u$, the projection $\Xi_u$ contains a translate
of the corresponding projection $K_u$.  Then there exists $x \in \RR^n$ such that 
$$K+x \subseteq D \subseteq \tfrac{n}{n-1}\Xi.$$
\label{simpcover}
\end{theorem}
For vectors $v,w \in \RR^n$
denote by $v|w^\perp$ the orthogonal projection of the vector $v$
onto the subspace $w^\perp$.
\begin{proof} Translate $K$ so that each coordinate plane $e_i^\perp$ supports $K$ on its
positive side.  In other words, $K$ is pushed into the corner of the positive orthant of $\RR^n$,
so that each 
$$h_K(-e_i) = 0.$$
Let $y \in K$, with coordinates $y = (y_1, \ldots, y_n)$.  The aforementioned repositioning of $K$
implies that each $y_i \geq 0$.  

We are given that each projection $\Xi_u$ contains a translate of $K_u$.  In particular, it follows that
there exists 
$$x = (0, x_2, \ldots, x_n) \in e_1^\perp$$
such that $K_{e_1} + x \subseteq \Xi_{e_1}$.
For each $i > 1$,
$$h_{K_{e_1}}(-e_i) + x \cdot (-e_i) = h_{K_{e_1}+x}(-e_i) \leq h_\Xi (-e_i) = 0.$$
Therefore, for each $i > 1$, we have 
$$0 \geq h_{K_{e_1}}(-e_i) + x \cdot (-e_i) = h_{K}(-e_i) + x \cdot (-e_i) = 0 - x_i,$$
so that each $x_i \geq 0$.

Since the coordinates of $x$ are non-negative, we have $x \cdot v \geq 0$, so that
$$h_{K_{e_1}}(v|e_1^\perp) = 
h_{K_{e_1}}(v) \leq h_{K_{e_1}}(v) + x \cdot v 
= h_{K_{e_1}+x}(v) \leq h_{\Xi_{e_1}}(v) = h_{\Xi_{e_1}}(v|e_1^\perp),$$
while $h_{K_{e_1}}(-e_i) = 0 = h_{\Xi_{e_1}}(-e_1)$ for each $i> 1.$  In other words, $K_{e_1}$ lies inside
each of the half-spaces of $e_1^\perp$ that define the simplex $\Xi_{e_1}$.  It follows that
$$K_{e_1} \subseteq \Xi_{e_1},$$
and we can use $x=o$.  Moreover, this argument applies in each of the directions $e_1, \ldots, e_n$, so that
$$K_{e_i} \subseteq \Xi_{e_i},$$
for each $i$.

Since 
$K_{e_j} \subseteq \Xi_{e_j}$ for each $j$, the width of $K$ is less than 1 in each 
coordinate direction $e_i$.  Since $K_{e_i} \subseteq \Xi_{e_i}$ as well, it follows that
$K \subseteq C_i$.  In other words,
$$K \subseteq \bigcap_{i=1}^n C_i = D.$$

The vertices of $D$ are $\{o, e_1, \ldots, e_n, p\}$.  Evidently, $o, e_1, \ldots, e_n \in \Xi \subseteq \tfrac{n}{n-1}\Xi$.
Since $p$ has positive coordinates which sum to $\frac{n}{n-1}$, we have $p \in \tfrac{n}{n-1}\Xi$ as well. It follows
that $D \subseteq \tfrac{n}{n-1}\Xi$.
\end{proof}

It immediately follows from Theorem~\ref{simpcover} that, if every projection $\Xi_u$ contains a translate
of the corresponding projection $K_u$, then
\begin{align}
V_n(K) \leq \left(\tfrac{n}{n-1}\right)^n V_n(\Xi).
\label{xi-bound}
\end{align}
Since $\left(\tfrac{n}{n-1}\right)^n$ decreases to $e$ as $n \rightarrow \infty$, this gives a universal
upper bound on the ratio
$$\frac{V_n(K)}{V_n(\Xi)}$$
under the condition of covering projections.  

The next proposition will allow us to generalize these observations from the special case of
covering by the simplex $\Xi$ to covering by an arbitrary $n$-simplex.
\begin{proposition} Let $K, L \in \KK_n$.  Let $\psi: \RR^n \rightarrow \RR^n$ be a
nonsingular linear transformation.  Then $L_u$ contains a translate of $K_u$ for all unit directions $u$
if and only if $(\psi L)_u$ contains a translate of $(\psi K)_u$~for all $u$.
\label{obl}
\end{proposition}
This proposition implies that nothing is gained (or lost) by allowing
more general (possibly non-orthogonal) linear projections.

\begin{proof} 
For $S \subseteq \RR^n$ and a nonzero vector $u$, let $\LL_S(u)$ denote the set of straight lines in $\RR^n$
parallel to $u$ and meeting the set $S$.
The projection $L_u$ contains a translate $K_u$ for each unit vector $u$ if and only if,
for each $u$, there exists $v_u$ such that
\begin{equation}
\LL_{K+v_u} (u) \subseteq \LL_L(u).
\label{hum}
\end{equation}
But $\LL_{K+v_u}(u) = \LL_{K}(u) + v_u$ 
and $\psi \LL_{K}(u) = \LL_{\psi K}(\psi u)$.
It follows that~(\ref{hum}) holds if and only if
$\LL_{K}(u) + v_u \subseteq \LL_L(u)$,
which, in turn, holds if and only if
$$\LL_{\psi K}( \psi u) + \psi v_u \subseteq \LL_{\psi L}(\psi u) \;\;\; \hbox{ for all unit } u.$$
Set 
$$\tilde{u} = \frac{\psi u}{|\psi u|} \;\;\; \hbox{ and } \;\;\; \tilde{v} = \psi v_u.$$
The relation~(\ref{hum}) now holds if and only if,
for each $\tilde{u}$, there exists $\tilde{v}$ such that
$$\LL_{\psi K}(\tilde{u}) + \tilde{v} \subseteq \LL_{\psi L}(\tilde{u}),$$
which holds if and only if 
$(\psi L)_{\tilde{u}}$ contains a translate of $(\psi K)_{\tilde{u}}$ for all $\tilde{u}$.
\end{proof}

Proposition~\ref{obl} implies that the projection covering relation is preserved by 
invertible affine transformations.  Since every $n$-dimensional simplex can be expressed as the affine image of
the simplex $\Xi$, the volume inequality~(\ref{xi-bound}) continues to hold when $\Xi$ is replaced by {\em any}
simplex whose projections cover those of $K$.  In the next section we will generalize this observation still further to
an even larger class of sets.  However, the volume bound~(\ref{xi-bound}) can also be strengthened for the special case
in which the projections of $K$ are covered by those
of a simplex.
\begin{theorem} Let $K \in \KK_n$ and let $T$ denote an 
$n$-dimensional simplex such that, for every unit vector $u$, the projection $T_u$ contains a translate
of the corresponding projection $K_u$.  Then
$$V_n(K) \leq \tfrac{n}{n-1} V_n(T).$$
\label{simpcover2}
\end{theorem}

\begin{proof} Without loss of generality (scaling as needed) we may assume that $V_n(T) = V_n(\Xi)$,
where $\Xi$ is the special simplex defined at the beginning of this section.
Applying volume-preserving affine transformations as needed, Proposition~\ref{obl} implies that 
we may also 
assume, without loss of generality, that
$T = \Xi$.  In this case the proof of Theorem~\ref{simpcover}
implies that $K$ lies inside the cap body $D$.  An elementary computation shows that
$$V_n(D) = \tfrac{n}{n-1}V_n(\Xi) = \tfrac{n}{n-1}V_n(T).$$
The theorem now follows from the monotonicity of volume.
\end{proof}

\section{When one body can hide behind another}
\label{sec.general}

We now re-state and prove the main theorem of this article, which
generalizes some of the results of the previous section to the case of {\em any}
two compact convex sets $K$, $L$ in $\RR^n$ such that orthogonal projections of $L$ contain translates of 
the corresponding 
projections of $K$.
\begin{theorem} Let $K$, $L \in \KK_n$.
Suppose that, for every unit vector $u$, the projection $L_u$ contains a translate
of the corresponding projection $K_u$.  Then there exists $x \in \RR^n$ such that
$$K + x \subseteq \tfrac{n}{n-1} L.$$
\label{coverbound}
\end{theorem}

This theorem gives a sharp bound for containment by sets with covering projections.  To see this, 
recall that if $K \in \KK_n$, the set $-nK$ contains a translate of $K$, by Proposition~\ref{KnK}.
Then consider the case in which $K$ is the regular unit edge $n$-simplex $\triangle$, 
and $L =(n-1)(-\triangle)$.  
For each direction $u$, the projection 
$-(n-1)\triangle_u$ contains a translate of $\triangle_u$.  
Meanwhile, the smallest dilate of $-\triangle$ to contain
a translate of $K=\triangle$ is 
$$n(-\triangle) = \tfrac{n}{n-1} (n-1)(-\triangle) = \tfrac{n}{n-1} L.$$
It follows that the coefficient $\tfrac{n}{n-1}$ in Theorem~\ref{coverbound} cannot be improved.

\begin{proof}[Proof of Theorem~\ref{coverbound}] 
Let $T$ be any $n$-simplex that contains $L$.  Since each projection $L_u$ contains
a translate of $K_u$, it follows that each $T_u$ contains a translate of $K_u$.  
Let $\Xi$ and $D$ again denote the simplex and cap body defined in the previous section. 
Let $\psi$ be an invertible affine transformation $\psi$
such that $\psi T = \Xi$.
By Proposition~\ref{obl} each projection $\Xi_u$ contains a translate of
the corresponding projection $(\psi K)_u$ of the body $\psi K$.
By Theorem~\ref{simpcover} there is $x \in \RR^n$ such that
$$\psi K + x \subseteq D \subseteq \tfrac{n}{n-1} \Xi.$$
Since $\psi$ is affine and invertible, it follows that the simplex
$$\tilde{T} = \tfrac{n}{n-1} T$$
contains a translate of $K$.  Meanwhile $\tilde{T}$ circumscribes 
$\tfrac{n}{n-1} L$ if and only if $T$ circumscribes $L$.
So we have shown that every circumscribing simplex $\tilde{T}$ of 
$\tfrac{n}{n-1} L$ contains a translate of $K$.  It follows
from the Lutwak Containment Theorem~\ref{lcon} that
$\tfrac{n}{n-1} L$ contains a translate of $K$.
\end{proof}

\begin{corollary} Let $K$, $L \in \KK_n$.
Suppose that, for every unit vector $u$, the projection $L_u$ contains a translate
of the corresponding projection $K_u$.  Then 
$$V_n(K) \leq \left( \tfrac{n}{n-1} \right)^n V_n(L).$$
\label{ebound}
\end{corollary}

\begin{proof} By Theorem~\ref{coverbound}
there exists $x \in \RR^n$ such that
$$K + x \subseteq \tfrac{n}{n-1} L,$$
so that
$$V_n(K) = V_n(K+x) \leq V_n \left( \tfrac{n}{n-1} L \right) = 
\left(  \tfrac{n}{n-1} \right)^n V_n(L).$$
\end{proof}

In particular, if each $L_u$ contains a translate of $K_u$ then there is a constant $c_n \in \RR$
independent of $K, L \in \KK_n$ such that
\begin{align}
V_n(K) \leq c_n V_n(L),
\label{cbound}
\end{align}
where $c_n \rightarrow e$ as $n \rightarrow \infty$.

The volume ratio bound~(\ref{cbound}) gives a substantial improvement over those previously known.
In \cite{Klain-Shadow}  
circumscribing cylinders were used to show that 
$c_n \leq n$ 
for all $n \geq 1$.
A simple argument also implies that $c_2 = 3/2$ is the best possible result in dimension $2$.
More generally,
an upper bound for $c_n$ can also be obtained using the Rogers-Shephard inequality 
(also known as the {\em difference body inequality} \cite{Chak-RS,Rogers-Shep}\cite[p. 409]{red}).
This inequality asserts that, for $K \in \KK_n$,
\begin{align}
V_n(K + (-K)) \leq \binom{2n}{n} V_n(K).
\label{rshep}
\end{align}
If $V_n(K) > 0$ then
equality holds in~(\ref{rshep}) if and only if $K$ is a simplex.

To obtain a bound for $c_n$ using~(\ref{rshep}), 
suppose $K, L \in \KK_n$ and that each $L_u$ contains a translate of $K_u$.
It follows that the same relation holds for the Minkowski symmetrals $\tfrac{1}{2}(L+ (-L))$ and 
$\tfrac{1}{2}(K+ (-K))$.  Since these symmetrals are both {\em centrally symmetric} it follows that
$$\tfrac{1}{2}(K+ (-K)) \subseteq \tfrac{1}{2}(L+ (-L)),$$
so that
$$V_n(K) \leq V_n \Big( \tfrac{1}{2}(K+ (-K)) \Big) 
\leq V_n \Big( \tfrac{1}{2}(L+ (-L)) \Big)
\leq \frac{1}{2^n} \binom{2n}{n} V_n(L),$$
where the first inequality follows from the Brunn-Minkowski inequality~(\ref{bmcc}) and the final
inequality is the Rogers-Shephard inequality~(\ref{rshep}).  In dimension 2 this yields 
the sharp bound $c_2 = 3/2$, where equality is attained when $L$ is a triangle
and $K = \tfrac{1}{2}(L+ (-L))$.  

However, for dimensions 3 and above, the Rogers-Shephard inequality no longer gives the best possible
bound for $c_n$.
More complicated results of Ball \cite{Ball-Shad} (see also \cite[p.163-164]{Gard2006}) 
imply that if the area of each projection of $L$ exceeds the corresponding
area of each projection of $K$ (a much weaker assumption than actual covering of projections) then $c_n$ grows with
order at most $\sqrt{n}$, with a universal (weak) bound of 
\begin{align}
V_n(K) \leq 1.1696 \sqrt{n} V_n(L).
\label{bbound}
\end{align}
The bound~(\ref{bbound}) implies that $c_3 \leq 2.026$, whereas the Rogers-Shephard bound only tells us
that $c_3 \leq 2.5$.  The inequality~(\ref{bbound}) still gives the best known bound for $c_n$ when $3 \leq n \leq 6$,
although all known numerical evidence suggests these bounds can be substantially improved 
(see also Section~\ref{sec.conclusion}).   Moreover, the bound~(\ref{bbound}) increases without limit 
as $n \rightarrow \infty$.

Theorem~\ref{coverbound} (and Corollary~\ref{ebound}) implies that, if projections of $L$ can cover
projections of $K$, then $c_n$ is actually bounded by a {\bf\em universal constant} independent of 
the dimension $n$.  
The inequality~(\ref{bbound}) implies that 
$$c_n \leq 1.1696 \sqrt{6} \approx 2.865,$$
for $n \leq 6$, 
but gives bounds larger than $3$ and increasing without limit for $n \geq 7$.
Meanwhile, Corollary~\ref{ebound} 
implies that
$$c_n \leq \Big( \tfrac{n}{n-1} \Big)^n 
\leq \Big( \tfrac{7}{6} \Big)^7 \approx 2.942$$
for $n \geq 7$,
giving a universal volume ratio bound of $c_n \leq 2.942$ in all finite dimensions.
Possible improvements for this universal bound are discussed in Section~\ref{sec.conclusion}.

\section{Projections to intermediate dimensions}
\label{sec.dim}

Theorem~\ref{coverbound} generalizes easily to projections onto an arbitrary lower dimension.
In order to obtain similar asymptotic bounds (as the ambient dimension $n \rightarrow \infty$),
these analogous results are best expressed
in terms of the {\em co-dimension} of the
projections.

If $\xi$ is a subspace of $\RR^n$ and $K \in \KK_n$,
we will denote by $K_\xi$ the orthogonal projection of $K$ onto $\xi$.
\begin{theorem} Let $K$, $L \in \KK_n$, and let $d \in \{1, \ldots, n-1\}$.
Suppose that, for every $(n-d)$-dimensional subspace $\xi \subseteq \RR^n$, the projection $L_\xi$ contains a translate
of the corresponding projection $K_\xi$.  Then there exists $x \in \RR^n$ such that
$$K + x \subseteq \tfrac{n}{n-d} L.$$
\label{coverboundi}
\end{theorem}

This theorem gives a sharp bound for containment by sets with covering projections.  To see this, 
recall that if $K \in \KK_n$, then each $(n-d)$-dimensional projection 
$(n-d)(-K)_\xi$ of $(n-d)(-K)$ contains a translate of the corresponding
projection $K_\xi$, by Proposition~\ref{KnK}.
Then consider the
case in which $K$ is the regular unit edge $n$-simplex $\triangle$, and $L =(n-d)(-\triangle)$.  For each 
$(n-d)$-dimensional subspace $\xi$,
the projection $(n-d)(-\triangle)_\xi$ contains a translate of $\triangle_\xi$.  Meanwhile, the smallest dilate of $-\triangle$ to contain
a translate of $K=\triangle$ is 
$$n(-\triangle) = \tfrac{n}{n-d} (n-d)(-\triangle) = \tfrac{n}{n-d} L.$$
It follows that the coefficient $\tfrac{n}{n-d}$ in Theorem~\ref{coverboundi} cannot be improved.

\begin{proof}[Proof of Theorem~\ref{coverboundi}] The case of $d = 1$ is addressed by Theorem~\ref{coverbound}.  
If $d > 1$ let $u \in \xi^\perp$ be a unit vector, and let $\overline{u}$ denote the line through the origin 
spanned by $u$.
By Theorem~\ref{coverbound}, applied within the $(n-d+1)$-dimensional space $\xi \oplus \overline{u}$, 
there is a vector $y$ such that
$$K_{\xi \oplus \overline{u}} + y \subseteq \tfrac{n-d+1}{n-d} L_{\xi \oplus \overline{u}}.$$
In other words, 
for every $(n-d+1)$-dimensional subspace $\xi' \subseteq \RR^n$
the set $\tfrac{n-d+1}{n-d} L_{\xi'}$ contains a translate of $K_{\xi'}$.  After $d$ iterations of this argument we obtain 
a vector $x$ such that
$$K + x \subseteq \tfrac{n}{n-1} \cdots \tfrac{n-d+2}{n-d+1}\tfrac{n-d+1}{n-d} L = \tfrac{n}{n-d}L.$$
\end{proof}

\begin{corollary} Let $K$, $L \in \KK_n$, and let $d \in \{1, \ldots, n-1\}$.
Suppose that, for every $(n-d)$-dimensional subspace $\xi \subseteq \RR^n$, the projection $L_\xi$ contains a translate
of the corresponding projection $K_\xi$. Then 
$$V_n(K) \leq \left( \tfrac{n}{n-d} \right)^n V_n(L).$$
\label{eboundi}
\end{corollary}

\begin{proof} By Theorem~\ref{coverboundi}
there exists $x \in \RR^n$ such that
$$K + x \subseteq \tfrac{n}{n-d} L,$$
so that
$$V_n(K) = V_n(K+x) \leq V_n \left( \tfrac{n}{n-d} L \right) = 
\left(  \tfrac{n}{n-d} \right)^n V_n(L).$$
\end{proof}

Note that, after fixing the co-dimension $d$, we have
$$\lim_{n\rightarrow\infty} \left( \tfrac{n}{n-d} \right)^n 
= \lim_{n\rightarrow\infty} \left(1+ \tfrac{d}{n-d} \right)^n
= \lim_{n\rightarrow\infty} \left(1+ \tfrac{d}{n-d} \right)^d \left(1+ \tfrac{d}{n-d} \right)^{n-d} = e^d.$$
Corollary~\ref{eboundi} implies that, if $L_\xi$ contains a translate of $K_\xi$ for every $(n-d)$-dimensional subspace $\xi \subseteq \RR^n$,
then there is a constant $c_{n,d} \in \RR$
independent of $K, L \in \KK_n$ such that
\begin{align}
V_n(K) \leq c_{n,d} V_n(L),
\label{cboundi}
\end{align}
where $c_{n,d} \rightarrow e^d$ as $n \rightarrow \infty$.
It follows that, for fixed co-dimension $d$, the coefficient $c_{n,d}$ can be replaced by
a universal constant $\gamma_d$ independent of the bodies $K$ and $L$ 
and independent of the ambient dimension $n$.

\section{Concluding remarks and open questions}
\label{sec.conclusion}

Numerical evidence suggests that
the volume ratio bounds in this article can almost certainly be improved \cite{cchen}.  
In Section~\ref{sec.general} we showed that, 
if each projection $L_u$ contains a translate of $K_u$, then
\begin{align}
V_n(K) \leq c V_n(L),
\label{c}
\end{align}
where $c$ is a constant independent of the dimension $n$, and where $c < 2.942$.
However, computational evidence suggests that $c$ is much smaller.  

If we fix the dimension $n$, then Theorem~\ref{simpcover2}
gives a value of $c = \frac{n}{n-1}$ when the body $L$ is an $n$-simplex.
On the other hand, previous work \cite{Klain-Shadow} implies that 
$$V_n(K) \leq V_n(L),$$
in the special case where $L$ is a cylinder, or even a cylinder body (that is,
a limit of Blaschke combinations of cylinders).  
This suggests that simplices may be the worst case scenario for bodies with covering
shadows, and motivates the following conjecture:
\begin{conjecture}
Let $K, L \in \KK_n$ and suppose that
$L_u$ contains a translate of $K_u$ 
for every unit vector $u \in \RR^n$.  
Then
\begin{align}
V_n(K) \leq \frac{n}{n-1} V_n(L),
\label{conj}
\end{align}
\end{conjecture}
This conjecture is already known to be true in dimension $2$ 
(indeed, we observed in Section~\ref{sec.general} that
$c_2 = 3/2$ 
is the best possible bound), but remains open for dimensions $n \geq 3$.
If this conjecture is true, then the universal volume ratio constant for all dimensions $n \geq 2$
would satisfy $c = 3/2$.

Even if the conjecture above is proven correct, it remains to determine 
the best upper bound for the ratio 
$$c_n = \frac{V_n(K)}{V_n(L)}$$ in each dimension
separately, for the conjectured bound~(\ref{conj}) does not appear to be sharp in dimensions $n \geq 3$.
Examples investigated so far suggest that the highest volume ratio ought to occur when a suitable convex Minkowski
combination of a simplex $\triangle$ with the scaled reflection $(n-1)(-\triangle)$ hides behind
the set $(n-1)(-\triangle)$.  

A direct computation \cite{cchen} shows that, if $\triangle$ is a tetrahedron in $\RR^3$, and if
\begin{align}
K = \left(1- \frac{1+\sqrt{56}}{11} \right)\triangle + \left(\frac{1+\sqrt{56}}{11} \right)(-2\triangle)
\quad \hbox{ and } \quad L = -2\triangle, 
\label{worst}
\end{align}
then each projection $L_u$ contains a translate of $K_u$ (by Proposition~\ref{KnK}, applied in dimension 2),
while
$$\frac{V_n(K)}{V_n(L)} \approx 1.1634.$$
See Figure~\ref{diff-b}.
\begin{figure}[ht]
\includegraphics[scale=0.1]{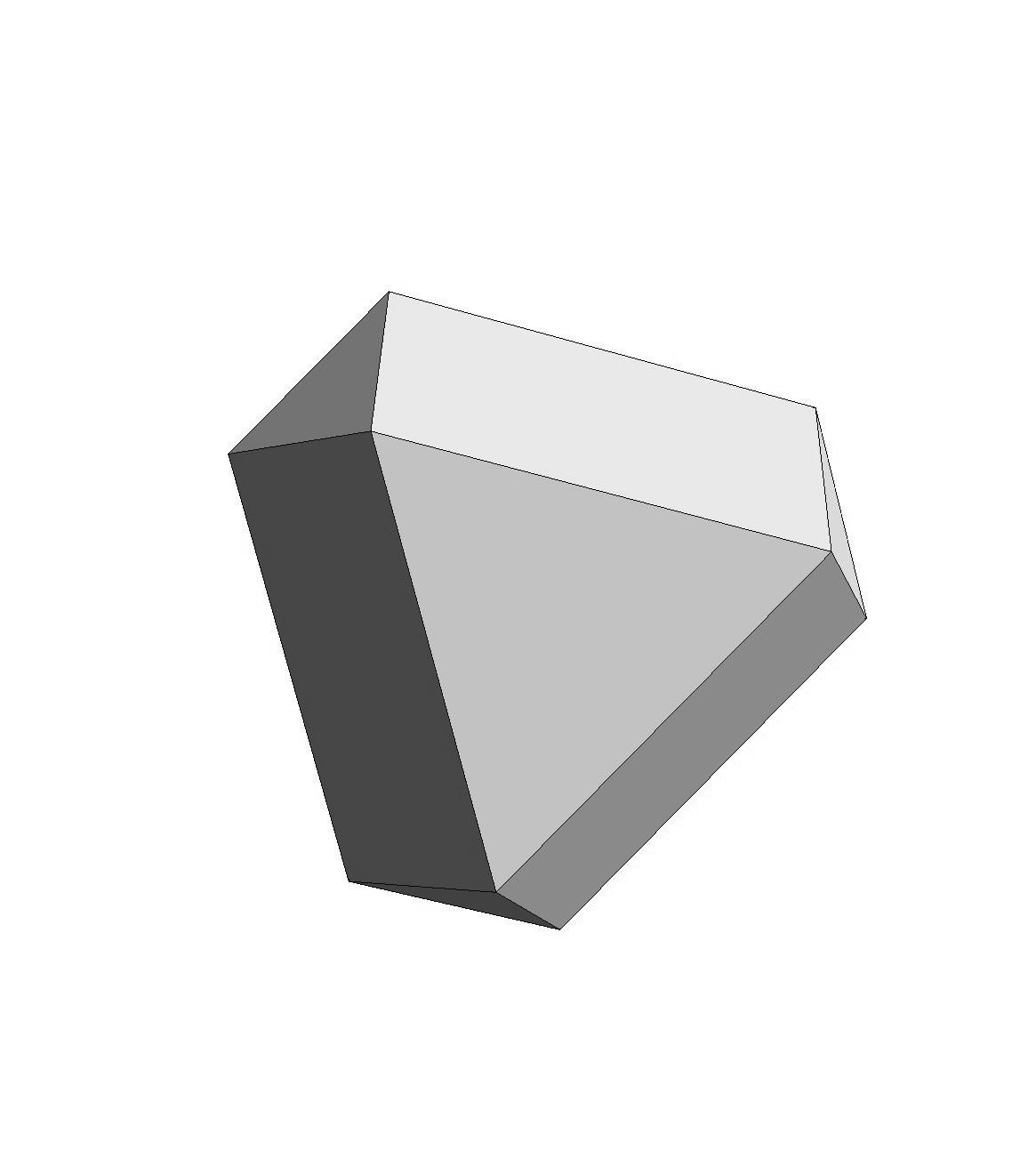} 
\includegraphics[scale=0.1]{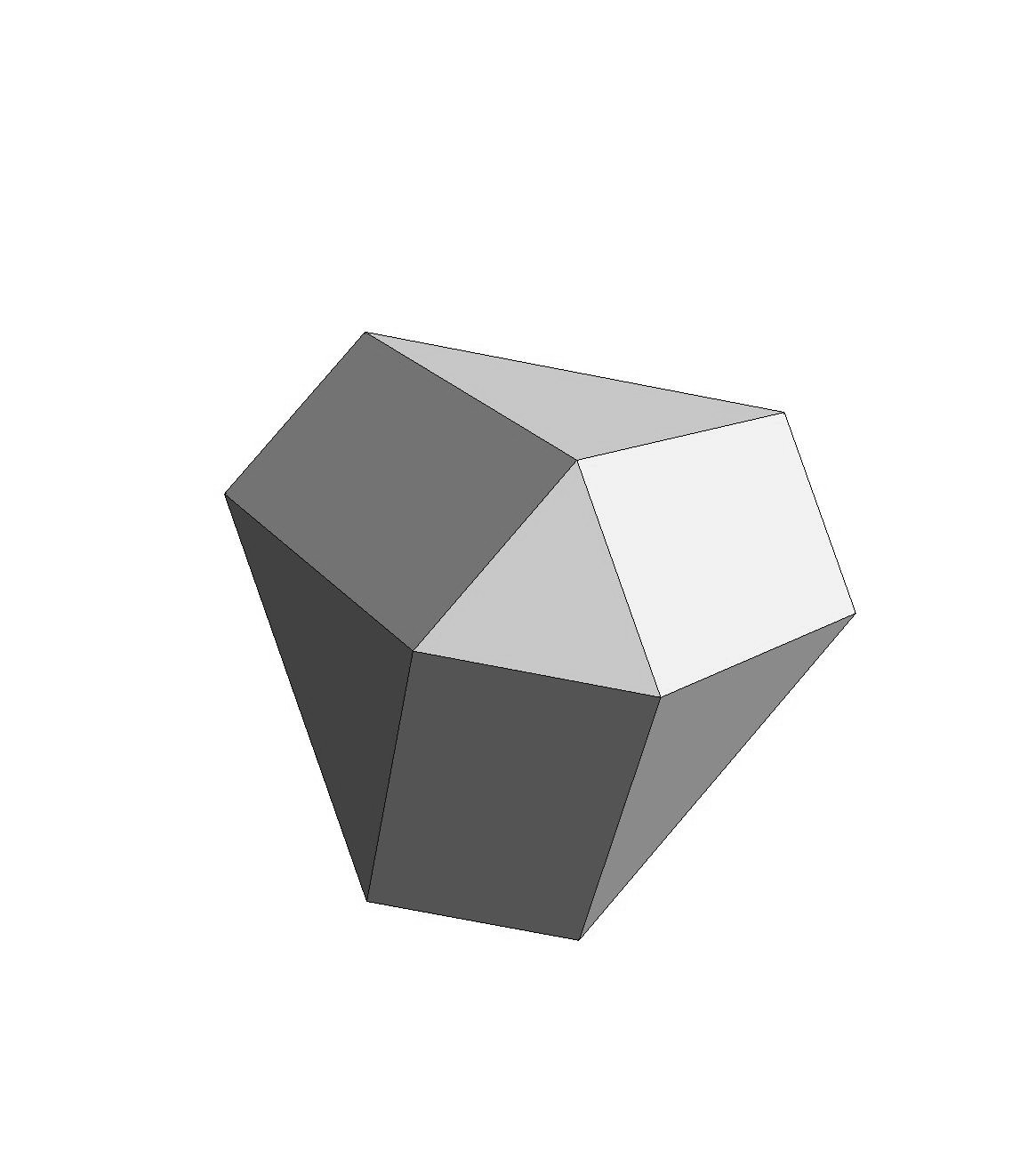} 
\caption{Two views of the Minkowski combination $K$ of a regular tetrahedron with its reflection, 
as specified in~(\ref{worst}).}
\label{diff-b}
\end{figure}
We conjecture that the best possible bound for the volume ratio in dimension 3 is the value $c_3 \approx 1.1634$ occurring
with the pair of bodies specified in~(\ref{worst}),
and that analogous computations with simplices 
in $\RR^n$ will yield the best bounds for $c_n$.  However, these
assertions remain conjectures at this point.

\begin{question}
What is the best possible value for the volume ratio bound $c_n$ for each particular
dimension $n \geq 3$?  For $c_{n,d}$?
\end{question}

\begin{question}
What is the best possible value for the universal volume ratio bound $c$ 
for all dimensions $n \geq 3$?  For $\gamma_{d}$?
\end{question}

\begin{question}
Let $K, L \in \KK_n$, and let $d \in \{ 1, \ldots, n-1 \}$.
Suppose that, for each $(n-d)$-dimensional subspace $\xi$ of $\RR^n$, the orthogonal projection 
$L_\xi$ of $K$ contains a translate of $K_\xi$.  

Under what simple (easy to state, easy to verify) additional conditions
does it follow that $V_n(K) \leq V_n(L)$?\\
\end{question}
Some partial answers to the third 
question are given  
in \cite{Klain-Shadow}.  
There it is shown that if $K_\xi$ can be translated inside $L_\xi$ 
for all $(n-d)$-dimensional subspaces $\xi$,  
then $K$ has smaller volume than $L$ whenever $L$ can be approximated by Blaschke combinations of
$(n-d)$-decomposable sets.  Moreover, in \cite{Klain-Circ} it is shown that, for example, 
if projections of a right square pyramid $Q$ contains translates of the projections of a convex body 
$K$ in $\RR^3$, then $Q$ contains a translate of $K$ (and so certainly has greater volume).  
Since $Q$ does not appear to be a cylinder body (a class of bodies not yet easily characterized), 
it is likely that a larger class of examples exist for bodies $L$ whose volume exceeds that of any
body $K$ having smaller shadows (up to translation).

The question of volume comparison was originally motivated by the following.
\begin{question}
Let $K, L \in \KK_n$, and let $d \in \{ 1, \ldots, n-1 \}$.
Suppose that, for each $(n-d)$-dimensional subspace $\xi$ of $\RR^n$, the orthogonal projection 
$L_\xi$ of $K$ contains a translate of $K_\xi$.  

Under what simple (easy to state, easy to verify) additional conditions
does it follow that $L$ contains a translate of $K$? \\
\end{question}
Some partial answers to this 
question are given in \cite{Klain-Inscr} and \cite{Klain-Circ}.

All of these many questions can be re-phrased allowing for 
(specified subgroups of) rotations (and reflections) as well as
translations.  However, the results obtained so far rely on the observation that the set of
translates of $K$ that fit inside $L$, that is, the set
$$\{v \in \RR^n \; | \; K+v \subseteq L\},$$
is itself a compact convex set in $\RR^n$.  
By contrast, the set of rigid motions
of $K$ that fit inside $L$ will lie in a more complicated Lie group.  For this reason (at least)
the questions of covering via rigid motions may be more difficult to address than
the case in which only translation is allowed.

\section*{Acknowledgements}

This project was supported in part by the 
Program for Research in Mathematics, Engineering, 
and Science for High School Students (PRIMES) at MIT.


\providecommand{\bysame}{\leavevmode\hbox to3em{\hrulefill}\thinspace}
\providecommand{\MR}{\relax\ifhmode\unskip\space\fi MR }
\providecommand{\MRhref}[2]{%
  \href{http://www.ams.org/mathscinet-getitem?mr=#1}{#2}
}
\providecommand{\href}[2]{#2}

\end{document}